\documentclass[11pt,a4paper]{amsart}
\usepackage{amssymb,color}
\usepackage[english]{babel}          
\usepackage{verbatim}
\usepackage[T1]{fontenc}             

\usepackage{caption}

\usepackage{floatflt,graphicx}
\usepackage{color,xcolor}
\usepackage{enumerate}
\usepackage{animate}
\usepackage{amsmath}
\usepackage{amsthm}

\newcounter{minutes}
\divide\time by 60
\newcounter{hours}
\multiply\time by 60 \addtocounter{minutes}{-\time}
\title
{Some remarks on the visual angle metric}

\author{Parisa Hariri}
\address{Department of Mathematics and Statistics\\
  University of Turku\\ Turku, Finland}
\email{parisa.hariri@utu.fi}

\author{Matti Vuorinen}
\address{Department of Mathematics and Statistics\\
  University of Turku\\ Turku, Finland}
\email{vuorinen@utu.fi}

\author{Gendi Wang}
\address{Department of Physics and Mathematics\\
  University of Eastern Finland\\ Joensuu, Finland}
\email{wgdwang@126.com}


\dedicatory{}

\theoremstyle{plain}
\newtheorem{thm}[equation]{Theorem}

\newtheorem{lem}[equation]{Lemma}
\newtheorem{example}[equation]{Example}
\newtheorem{prop}[equation]{Proposition}

\theoremstyle{definition}

\theoremstyle{remark}
\newtheorem{rem}[equation]{Remark}
\newtheorem{conj}[equation]{Conjecture}
\newtheorem{nonsec}[equation]{}

\newtheorem{algor}[equation]{Algorithm}
\numberwithin{equation}{section}
\newcommand{\diam}{\ensuremath{\textrm{diam}}}

\newcommand{\beq}{\begin{equation}}
\newcommand{\eeq}{\end{equation}}

\newcommand{\bequu}{\begin{eqnarray*}}
\newcommand{\eequu}{\end{eqnarray*}}

\newcommand{\bequ}{\begin{eqnarray}}
\newcommand{\eequ}{\end{eqnarray}}

\newcommand{\B}{\mathbb{B}^2}
\newcommand{\R}{\mathbb{R}^2}

\newcommand{\UH}{\mathbb{H}^2}
\newcommand{\BB}{\mathbb{B}^2}

\newcommand{\Hn}{ {\mathbb{H}^n} }
\newcommand{\Bn}{ {\mathbb{B}^n} }
\newcommand{\Rn}{ {\mathbb{R}^n} }

\newcommand{\sh}{\,\textnormal{sh}}
\newcommand{\ch}{\,\textnormal{ch}}
\renewcommand{\th}{\,\textnormal{th}}

\renewcommand{\Im}{{ \rm Im}\,}
\renewcommand{\Re}{{ \rm Re}\,}

\begin{document}

\def\thefootnote{}
\footnotetext{ \texttt{File:~\jobname .tex,
           printed: \number\year-\number\month-\number\day,
           \thehours.\ifnum\theminutes<10{0}\fi\theminutes}
} \makeatletter\def\thefootnote{\@arabic\c@footnote}\makeatother

\begin{abstract}
We show that the visual angle metric and the triangular ratio metric are comparable in convex domains. We also find the extremal points for the visual angle metric in the half space and in the ball by use of a construction based on hyperbolic geometry. Furthermore, we study distortion properties of quasiconformal maps with respect to the triangular ratio metric and the visual angle metric.
\end{abstract}

\maketitle


\section{Introduction}\label{section 1}


\setcounter{equation}{0}
Geometric function theory studies classes of mappings between subdomains of the Euclidean space $\mathbb{R}^n,\, n\geq 2$. These classes include both injective and noninjective mappings. In particular, Lipschitz, quasiconformal, and quasiregular mappings along with their generalizations such as maps with integrable dilatation are in the focus. On the other hand, this theory has also been extended to Banach spaces and even to metric spaces. What is common to these theories is that various metrics are extensively used as powerful tools, e.g., V\"{a}is\"{a}l\"{a}'s theory of quasiconformality in Banach spaces \cite{v2} is based on the study of metrics: the norm metric, the quasihyperbolic metric and the distance ratio metric. In recent years several authors have studied the geometries defined by these and other related metrics \cite{h,himps,kl,mv,rt}. For a survey of these topics the reader is referred to \cite{vu2}.

The main purpose of this paper is to continue the study of some of these metrics.
For a domain $G \subsetneq \mathbb{R}^n$ and $x, y\in G$, {\it the visual angle metric} is defined by
\begin{equation}\label{v}
v_G(x,y)=\sup\{\measuredangle (x,z,y): z\in\partial{G} \}\in [0,\pi]\,,
\end{equation}
where $\partial{G}$ is not a proper subset of a line. This metric was introduced and studied very recently in \cite{klvw}.
It is clear that a point $z\in\partial{G}$ exists for which this supremum is attained, such a point $z$ is called {\it an extremal point for $v_G(x,y)$}.
For a domain $G \subsetneq \mathbb{R}^n$ and $x,y \in G$, {\it the triangular ratio metric} is defined by
\begin{equation}\label{sm}
s_G(x,y)=\sup_{z\in \partial G}\frac{|x-y|}{|x-z|+|z-y|}\in [0,1].
\end{equation}
Again, the existence of an extremal boundary point is obvious. This metric has been studied in \cite{chkv,hklv}.
The above two metrics are
closely related, for instance, both depend on  extremal boundary points. It is easy to see that both metrics are monotone with respect to domain. Thus if $D$ and $G$ are domains in $\Rn$ and $D\subset G$ then for all $x, y\in D$, we have
\[
s_D(x,y)\geq s_G(x,y), \quad v_D(x,y)\geq v_G(x,y).
\]
On the other hand, we will see that these two metrics are not comparable in some domains.

The paper is organized into sections as below. In section $2$ some comparisons between the visual angle metric and the triangular ratio metric in convex domains are given. In section $3$ we find the extremal points for the visual angle metric in the half space and in the ball by use of a construction based on hyperbolic geometry.  Our main results are given in section $4$, where uniform continuity of quasiconformal maps with respect to the triangular ratio metric and the visual angle metric is studied.


\section{Notation and preliminaries}\label{section 2}

In this section, we compare the visual angle metric and the triangular ratio metric in the convex domains and we also show that these two metrics are not comparable in $\BB\setminus\{0\}$.

Given two points $x$ and $y$ in $\Rn$, the segment between them is denoted by
$$
[x,y]=\{(1-t)x+ty\;:\; 0\le t\le1\}\,.
$$

Given three distinct points $x\,,y\,,z \in \Rn$, the notation $\measuredangle(x,z,y)$
means the angle in the range $[0,\pi]$ between the
segments $[x,z]$ and $[y,z]$.

For a domain $G$ of $\overline{\Rn}$, let $\mbox{M\"ob}(G)$ be the group of all M\"obius transformations which map $G$ onto itself.

\begin{nonsec}{\bf Hyperbolic metric.}
The hyperbolic metric $\rho_{\mathbb{H}^n}$ and $\rho_{\mathbb{B}^n}$ of the upper
half space ${\mathbb{H}^n} = \{ (x_1,\ldots,x_n)\in {\mathbb{R}^n}:  x_n>0 \} $
and of the unit ball ${\mathbb{B}^n}= \{ z\in {\mathbb{R}^n}: |z|<1 \} $ can be defined as weighted metrics with the weight functions
 $w_{\mathbb{H}^n}(x)=1/{x_n}$ and
 $w_{\mathbb{B}^n}(x)=2/(1-|x|^2)\,,$ respectively. This definition as such
is rather abstract and for applications explicit formulas
are needed. By \cite[p.35]{b} we have
\begin{equation}\label{cro}
\ch{\rho_{\mathbb{H}^n}(x,y)}=1+\frac{|x-y|^2}{2x_ny_n}\,,
\end{equation}
for all $x,y\in \mathbb{H}^n$, and by \cite[p.40]{b} we have
\begin{equation}\label{sro}
\sh{\frac{\rho_{\mathbb{B}^n}(x,y)}{2}}=\frac{|x-y|}{\sqrt{(1-|x|^2)(1-|y|^2)}}\,,
\end{equation}
and 
\begin{equation}\label{thrho}
\th{\frac{\rho_{\mathbb{B}^n}(x,y)}{2}}=\frac{|x-y|}{\sqrt{|x-y|^2+(1-|x|^2)(1-|y|^2)}}\,,
\end{equation}
for all $x,y\in \Bn$.
\end{nonsec}

Hyperbolic geodesic lines are arcs of circles which are orthogonal to the boundary of the domain.

\begin{nonsec}{\bf Distance ratio metric.}
For a proper open subset $G$ of $\Rn$ and for all
$x,y\in G$, the  distance ratio
metric $j_G$ is defined as
\begin{eqnarray*}
 j_G(x,y)=\log \left( 1+\frac{|x-y|}{\min \{d(x,\partial G),d(y, \partial G) \} } \right)\,.
\end{eqnarray*}
The distance ratio metric was introduced by Gehring and Palka
\cite{gp} and in the above simplified form by  Vuorinen \cite{vu1}. Both definitions are
frequently used in the study of hyperbolic type metrics \cite{himps}
and geometric theory of functions.
\end{nonsec}

By \cite[Theorem 3.8]{klvw} and \cite[Lemma 7.56]{avv},
\begin{equation}\label{vrhoj}
v_{\mathbb{B}^n}(x,y)\le\rho_{\mathbb{B}^n}(x,y)\le 2j_{\mathbb{B}^n}(x,y)\,,\,{\rm for\,\, all}\,\, x\,,y\in\Bn.
\end{equation}


The triangular ratio metric and the hyperbolic metric satisfy the following inequality in the unit ball \cite[Lemma 3.4, Lemma 3.8]{chkv} and \cite[Theorem 3.22]{chkv}:
\begin{equation}\label{srhochkv}
\frac{1}{2} \th{\frac{\rho_{\mathbb{B}^n}(x,y)}{2}}\le s_{\mathbb{B}^n}(x,y)\leq \th{\frac{\rho_{\mathbb{B}^n}(x,y)}{2}}\,,\, {\rm for\,\, all}\,\, x\,,y\in\Bn.
\end{equation}

One of the main results in \cite{klvw} is the following relation between the visual angle metric and the hyperbolic metric:
for all $x\,,y\in \Bn$,
\begin{equation}\label{gen1}
\arctan\left({\rm sh}\frac{\rho_\Bn(x,y)}{2}\right)\leq v_\Bn(x,y)\leq 2\arctan\left({\rm sh}\frac{\rho_\Bn(x,y)}{2}\right)\,,
\end{equation}
see \cite[Theorem 3.11]{klvw}.

If $0\,,x$ and $y$ are collinear or one of the two points $x$ and $y$ is $0$, then by \cite[Lemma 3.10]{klvw} and \eqref{sro},
\begin{eqnarray}\label{gen}
v_{\Bn}(x,y)=\arctan{\left(\frac{|x-y|}{\sqrt{(1-|x|^2)(1-|y|^2)}}\right)}.
\end{eqnarray}
If $x\,,y\in\Hn$ are located on a line orthogonal to $\partial\Hn$, then by \cite[Lemma 3.18]{klvw} and \eqref{cro}
\begin{eqnarray}\label{gen2}
v_{\Hn}(x,y)=\arctan{\left(\frac{|x-y|}{2\sqrt{x_ny_n}}\right)}.
\end{eqnarray}

By the monotonicity of the function $x\mapsto (\arctan x)/x$, it is easy to see that
\begin{equation}\label{ineq:th}
\frac{\pi}4 x\le\arctan x\le x\,,\,\,\forall\, x\in(0,1)\,.
\end{equation}

Our next goal is to prove Theorem \ref{2.16}. Our original argument gave the result with the weaker constant $8$ in place of $\pi$, for the unit ball. The present version is based on the suggestions of the referee, who also suggested the following two Lemmas. We would like to point out that \eqref{01} was already proved in our manuscript \cite{hvz}, which was written shortly after the first version of the present paper.

\begin{lem}\label{2.12}
Let $D\subsetneq \Rn$ be a domain. Then for $x, y\in D$
\beq\label{01}
\sin{\frac{v_D(x,y)}{2}}\leq s_D(x,y),
\eeq
\beq\label{02}
v_D(x,y)\leq \pi s_D(x,y).
\eeq

\end{lem}

\begin{proof}
Fix $\theta\in (0,\pi)$ such that
\beq\label{1l}
s_D(x,y)=\sin{ \frac{\theta}{2}}.
\eeq
Then the ellipsoid 
\[
E=\{z\in \Rn: |z-x|+|z-y|<|x-y|/\sin{\frac{\theta}{2}}\}
\]
is contained in $D$. Hence by domain monotonicity 
\beq\label{2l}
v_D(x,y)\leq v_E(x,y)=\theta,
\eeq
and by \eqref{1l} and \eqref{2l} we see that $\sin{\frac{v_D(x,y)}{2}}\leq s_D(x,y).$ Moreover, 
\[
\frac{v_D(x,y)}{s_D(x,y)}\leq \frac{\theta}{\sin(\theta/2)}\leq \pi
\]
where the last inequality follows from Jordan's inequality.
\end{proof}

\begin{figure}[ht]\label{lem2.12}
\begin{center}
     \includegraphics[width=11cm]{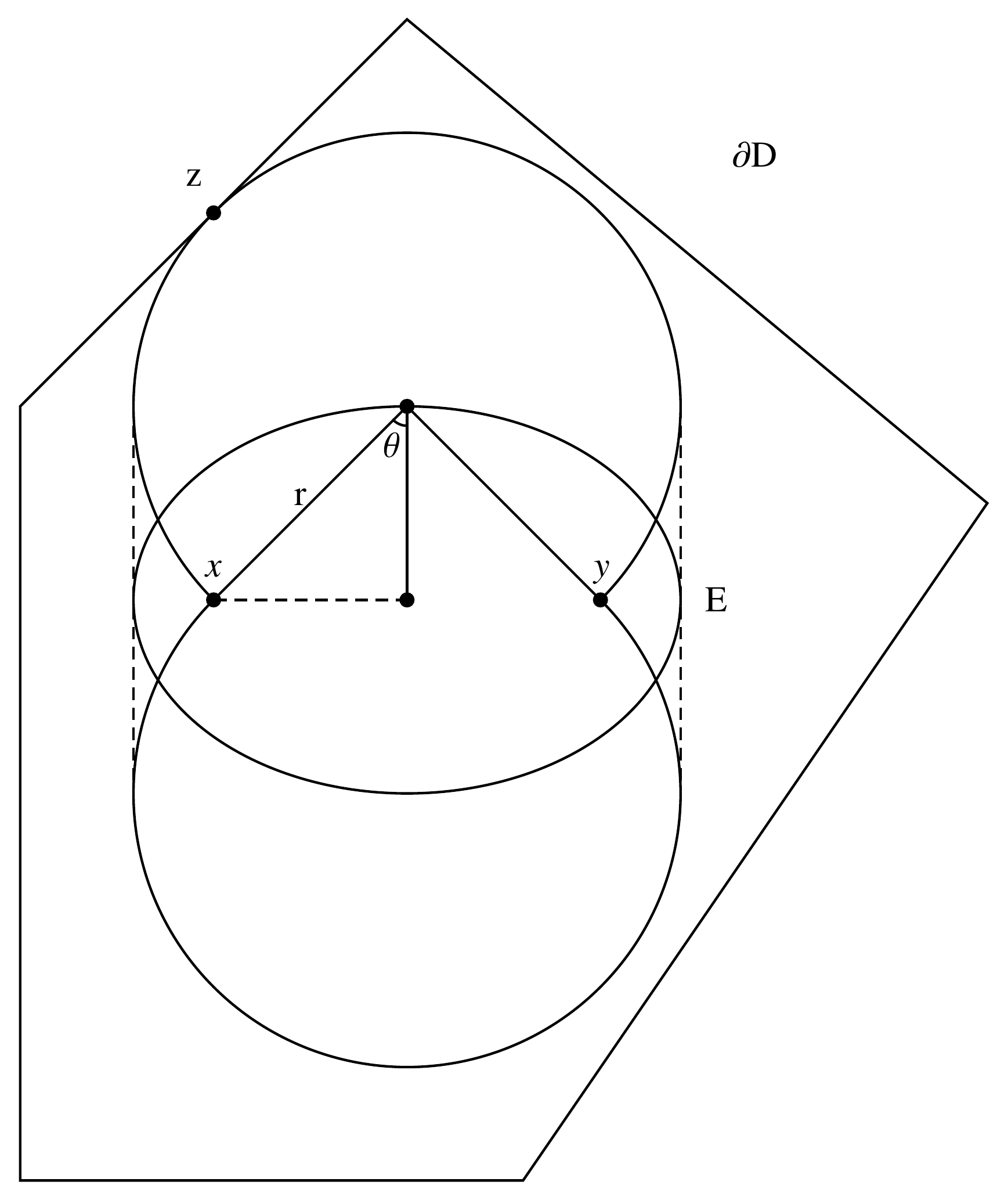}
\caption{Proof of Lemma \ref{2.15}, $v_D(x,y)=\theta=\measuredangle (x,z,y)$.}
\end{center}
\end{figure}

\begin{lem}\label{2.15}
If D is a convex subdomain of $\Rn$, and $x, y\in D$, then 
\[
s_D(x,y)\leq v_D(x,y).
\]
\end{lem}
\begin{proof}
If $v_D(x,y)\geq \frac{\pi}{2}$, then $s_D(x,y)\leq 1\leq v_D(x,y)$. Assume that $v_D(x,y)< \frac{\pi}{2}$. Let $z\in\partial{D}$ be such that $v_D(x,y)=\measuredangle (x,z,y)$. 
Denote by $r$ the radius of the circle passing through $x, y, z$. Since $D$ is convex, the convex hull of the set $\{z\in\Rn\, :\, \measuredangle (x,z,y)> v_D(x,y) \}$ is contained in $D$ and hence also the ellipsoid
\[
E=\{z\in \Rn: |z-x|+|z-y|<2r\}
\]
is contained in the domain $D$, see Figure 1. We easily conclude that $\sin{v_D(x,y)}=\frac{|x-y|}{2r}$. Moreover, by the domain monotonicity property of the $s$- metric, 
\[
s_D(x,y)\leq s_E(x,y)=\frac{|x-y|}{2r}=\sin{v_D(x,y)}\leq v_D(x,y).
\]
\end{proof}

\begin{thm}\label{2.16}
Let $D\subsetneq \Rn$ be a convex domain. Then for all $x,y\in D$ we have 
\[
s_D(x,y)\leq v_D(x,y)\leq \pi s_D(x,y).
\]
\end{thm}

\begin{proof}
The proof follows from Lemmas \ref{2.12} and \ref{2.15}.
\end{proof}

\begin{rem} The visual angle metric and the triangular ratio metric both highly depend on the boundary of the domain. If we replace the convex domain $D$ in Theorem \ref{2.16} with $G=\BB\setminus\{0\}$, then the visual angle metric and the triangular ratio metric are not comparable in $G$.
To this end, we consider two sequences of points $x_k=(1/k,0),\, y_k=(1/{k^2},0)\, (k=2,3,\cdots)$. Then
$$s_G(x_k,y_k)=\frac{k-1}{k+1}\,.$$
By \eqref{gen}, we get
\bequu
v_G(x_k,y_k)&=&v_{\BB}(x_k,y_k)\\
&=& \arctan\left(\frac{k}{(k+1)\sqrt{1+k^2}}\right) < \arctan\left(\frac{1}{k+1}\right).
\eequu
Therefore,
$$\frac{s_G(x_k,y_k)}{v_G(x_k,y_k)}\geq \frac{k-1}{(k+1)\arctan{\left(\frac{1}{k+1}\right)}}\rightarrow \infty\,,\,\,{\rm as}\,\,k\rightarrow\infty\,.$$

\end{rem}

\medskip

\section{The extremal points for the visual angle metric}\label{section 3}


In this section we aim to find the extremal points for the visual angle metric in the half space and in the ball by use of a construction based on hyperbolic geometry. Since the visual angle metric is similarity invariant we can consider it in the upper half space and in the unit ball.

\begin{thm}
Given two distinct points $x\,,y\in\Hn$, let $J[x,y]$ be the hyperbolic segment joining $x$ and $y$. Let $L_{xy}$ be the hyperbolic bisector of $J[x,y]$ with two endpoints $u$ and $w$ in $\partial\Hn$. Then one of the endpoints $u$ and $w$ is the extremal point for $v_{\Hn}(x,y)$, specifically,\\
(i) if one of $u$ and $w$ is infinity, say $w=\infty$, then $v_{\Hn}(x,y)=\measuredangle(x,u,y)$;\\
(ii) if none of $u$ and $w$ is infinity, then $v_{\Hn}(x,y)=\max\{\measuredangle(x,u,y), \measuredangle(x,w,y)\}$.
\end{thm}

\begin{proof}
It suffices to consider the $2$-dimensional case. We divide the proof into two cases.

{\bf Case 1.} $d(x,\partial\UH)=d(y,\partial\UH)$.

In this case, the hyperbolic bisector $L_{xy}$ of the hyperbolic segment $J[x,y]$ is also the bisector of the Euclidean segment $[x,y]$. We may assume that $u=L_{xy}\cap\partial\UH$. Then by simple geometric observation, we see that
$$v_{\UH}(x,y)=\measuredangle(x,u,y)\,.$$

\medskip
\begin{figure}[h]
\begin{minipage}[t]{0.45\linewidth}
\centering
\includegraphics[width=7cm]{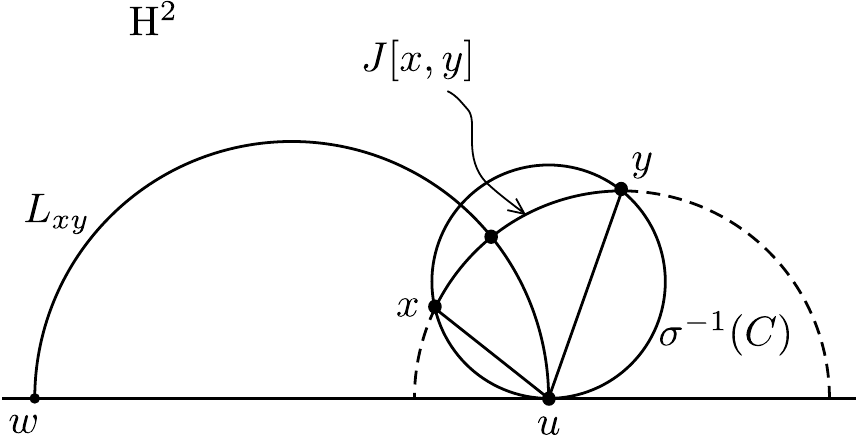}
\caption{\label{a1}}
\end{minipage}
\hfill
\hspace{1cm}
\begin{minipage}[t]{0.45\linewidth}
\centering
\includegraphics[width=5.3cm]{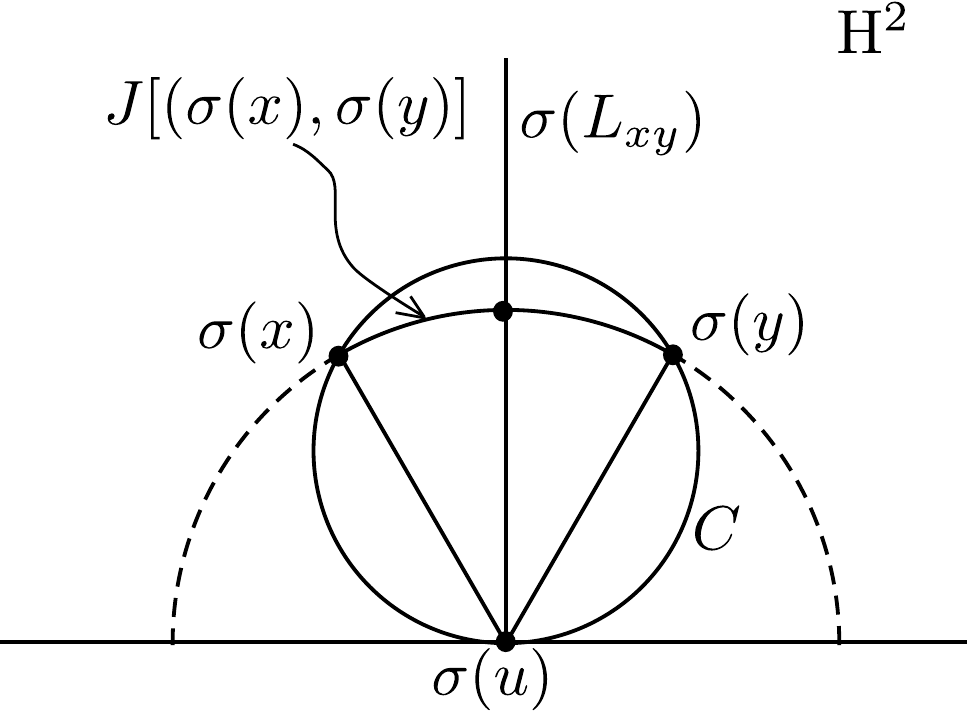}
\caption{\label{a2}}
\end{minipage}
\caption*{The M\"obius transformation $\sigma\in\mbox{M\"ob}(\UH)$ with $\sigma(w)=\infty$ maps Figure \ref{a1} onto  Figure \ref{a2}. Here $v_{\UH}(x,y)=\measuredangle(x,u,y)$.}
\end{figure}
\medskip

{\bf Case 2.} $d(x,\partial\UH)\neq d(y,\partial\UH)$.

Without loss of generality, we may assume that $d(x,\partial\UH)<d(y,\partial\UH)$ and that $u$ is located on the diameter of the circle containing $J[x,y]$. Then
$$\max\{\measuredangle(x,u,y), \measuredangle(x,w,y)\}=\measuredangle(x,u,y)\,.$$
By using a M\"obius transformation $\sigma\in\mbox{M\"ob}(\UH)$ with $\sigma(w)=\infty$, we see that $\sigma(L_{xy})$ is a Euclidean line orthogonal to $\partial\UH$ and also the hyperbolic bisector of the hyperbolic segment $J[\sigma(x),\sigma(y)]$. Hence $d(\sigma(x),\partial\UH)=d(\sigma(y),\partial\UH)$.  By the argument of Case 1, we have that $v_{\UH}(\sigma(x),\sigma(y))=\measuredangle(\sigma(x),\sigma(u),\sigma(y))$. Therefore, it is clear that the circle $C$ through $\sigma(x)$, $\sigma(u)$, $\sigma(y)$ is tangent to $\partial\UH$. Because M\"obius transformations preserve circles, we conclude that the circle $\sigma^{-1}(C)$ through $x,u,y$ is also tangent to $\partial\UH$, which implies that
$$v_{\UH}(x,y)=\measuredangle(x,u,y)\,.$$
This completes the proof.
\end{proof}

 In a similar way, we have the following conclusion for the visual angle metric in the unit ball.

\begin{thm}
Given two distinct points $x\,,y\in\Bn$, let $J[x,y]$ be the hyperbolic segment joining $x$ and $y$. Let $L_{xy}$ be the hyperbolic bisector of $J[x,y]$ with two endpoints $u$ and $w$ in $\partial\Bn$. Then one of the endpoints $u$ and $w$ is the extremal point for $v_{\Bn}(x,y)$, specifically,\\
(i) if $x$ and $y$ are symmetric with respect to the origin $0$, then $v_{\Bn}(x,y)=\measuredangle(x,u,y)=\measuredangle(x,w,y)$;\\
(ii) if $x$ and $y$ are not symmetric with respect to the origin $0$, then $v_{\Bn}(x,y)=\max\{\measuredangle(x,u,y), \measuredangle(x,w,y)\}$.
\end{thm}

\begin{proof}
It suffices to consider the $2$-dimensional case. We divide the proof into two cases.

{\bf Case 1.} The points $x$ and $y$ are symmetric with respect to the origin $0$.

In this case, the hyperbolic bisector $L_{xy}$ of the hyperbolic segment $J[x,y]$ is also the bisector of the Euclidean segment $[x,y]$. Then by simple geometric observation, we see that
$$v_{\Bn}(x,y)=\measuredangle(x,u,y)=\measuredangle(x,w,y)\,.$$

\begin{figure}[h]
\begin{minipage}[t]{0.44\linewidth}
\centering
\includegraphics[width=5cm]{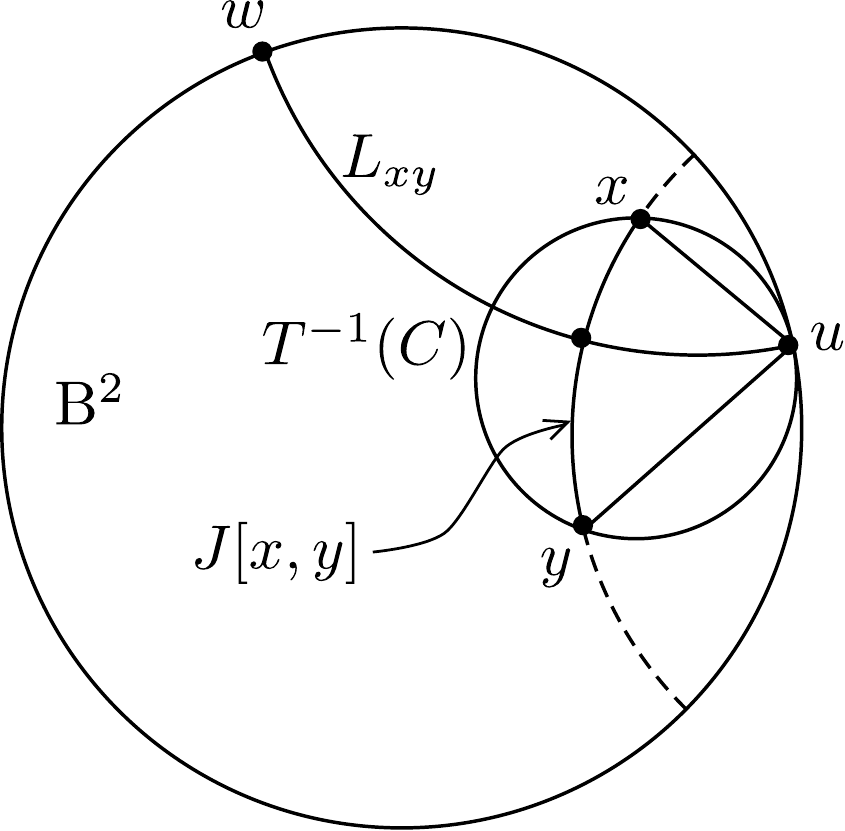}
\caption{\label{a3}}
\end{minipage}
\hfill
\hspace{1.2cm}
\begin{minipage}[t]{0.44\linewidth}
\centering
\includegraphics[width=4.5cm]{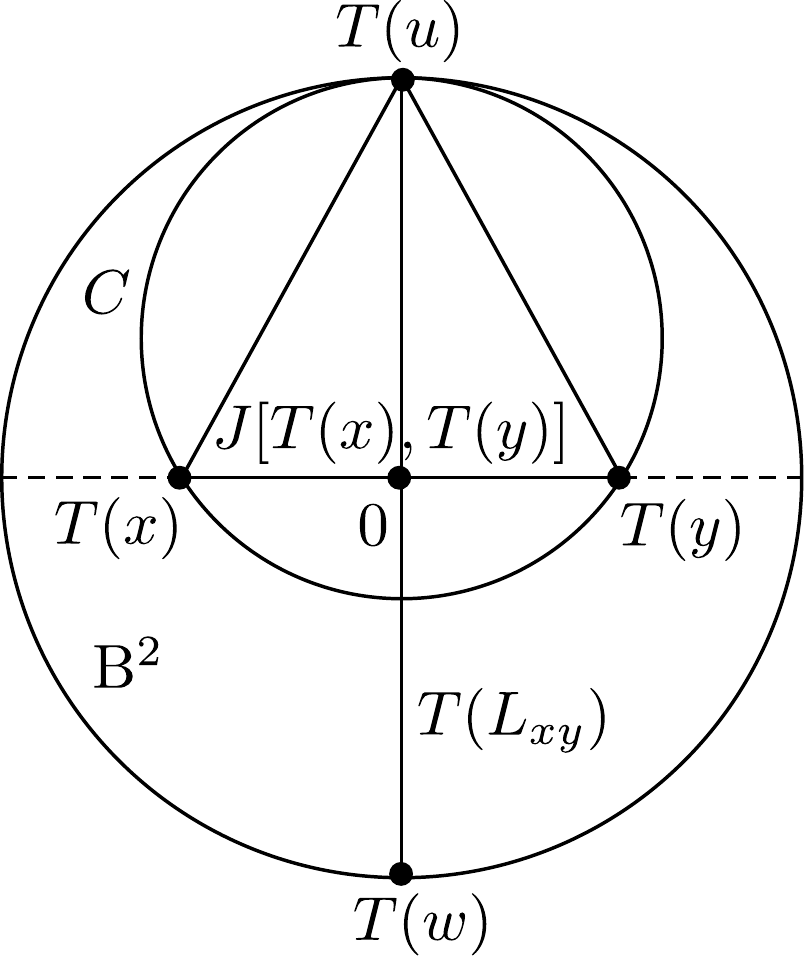}
\caption{\label{a4}}
\end{minipage}
\caption*{The M\"obius transformation $T\in\mbox{M\"ob}(\B)$ with $T(x)=-T(y)$ maps Figure \ref{a3} onto  Figure \ref{a4}. Here $v_{\B}(x,y)=\measuredangle(x,u,y)$.}
\end{figure}

{\bf Case 2.} The two points $x$ and $y$ are not symmetric with respect to the origin $0$.

The hyperbolic geodesic line through $x$ and $y$ divides $\partial\B$ into two arcs. Without loss of generality, we may assume that $u$ is in the minor arc or the semicircle (if $x,0,y$ are collinear) of $\partial\B$. Then
$$\max\{\measuredangle(x,u,y), \measuredangle(x,w,y)\}=\measuredangle(x,u,y)\,.$$
By using a M\"obius transformation $T\in\mbox{M\"ob}(\B)$ with $T(x)=-T(y)$, we see that $T(L_{xy})$ is the hyperbolic bisector of the hyperbolic segment $J[T(x),T(y)]$. By the argument of Case 1, we have that $v_{\B}(T(x),T(y))=\measuredangle(T(x),T(u),T(y))$. Therefore, it is clear that the circle $C$ through $T(x)$, $T(u)$, $T(y)$ is tangent to $\partial\B$. Because M\"obius transformations preserve circles, we conclude that the circle $T^{-1}(C)$ through $x,u,y$ is also tangent to $\partial\B$, which implies that
$$v_{\B}(x,y)=\measuredangle(x,u,y)\,.$$
This completes the proof.
\end{proof}

\begin{rem}
In \cite{vw}, the authors presented several methods of geometric construction to find the hyperbolic midpoint of a hyperbolic segment only based on Euclidean compass and ruler. These methods of construction can be used to find the extremal points for the visual angle metric in the upper half space and in the unit ball as the above theorems show.
\end{rem}

\bigskip


\section{Uniform continuity of quasiconformal maps}\label{section 4}

In this section, we study the uniform continuity of quasiconformal maps with respect to the triangular ratio metric and the visual angle metric.

We use notation and terminology from \cite{avv} in the sequel. We always take the boundary of a domain in $\Rn$ with respect to $\overline{\Rn}$ in this section. Let $\gamma_n$ and $\tau_n$ be the conformal capacities of the $n$-dimensional Gr\"otzsch ring and Teichm\"uller ring, respectively. Both of $\gamma_n$ and $\tau_n$ are continuous and strictly decreasing, see \cite[(8.34), Theorem 8.37]{avv}. Let $\lambda_n\in [4, 2e^{n-1}]$
be the Gr\"otzsch ring constant, see \cite[(8.38)]{avv}.

For $K>0$, the distortion function $\varphi_{K,n}(r)$ is a self-homeomorphism of $(0,1)$ defined by \cite[(8.69)]{avv}
$$\varphi_{K,n}(r)=\frac{1}{\gamma_n^{-1}(K\gamma_n(1/r))}\,.$$
For $K\geq 1$, $n\geq 2$, and $r\in (0,1)$,
\begin{equation}\label{phikn}
\varphi_{K,n}(r)\leq \lambda_n^{1-\alpha}r^{\alpha}\,,\,\,\alpha=K^{1/(1-n)}\,,
\end{equation}
see \cite[Theorem 7.47]{vu1}.

For $K\geq 1$, $n\geq 2$, and $t\in (0,1)$, the function $\Theta_{K,n}(t)$ in \cite[(2.9)]{fmv} is defined by:
\begin{equation}\label{ttg}
\Theta_{K,n}(t)=\frac{1}{\tau_n^{-1}(K\gamma_n(1/t))}=\frac{x^2}{1-x^2},\quad x=\varphi_{2^{n-1}K,n}(t)\,.
\end{equation}

If $\partial G$ has positive capacity, it can be shown that the conformal invariant $\mu_G$ is a metric in the domain $G$; this metric is called {\it the modulus metric}, see \cite[8.80]{avv}.
If $D$ is a subdomain of $G$, then $\mu_G(x,y)\le\mu_D(x,y)$ for all $x\,,y\in D$, see \cite[Remarks 8.83(2)]{avv}.
A $K$-quasiconformal map $f:G\rightarrow fG$ is $K$-bilipschitz in the $\mu_G$ metric, see \cite[(16.11)]{avv}.

\begin{lem}\label{lem:8.86avv}\cite[Lemma 8.86]{avv}
Let $G$ be a proper subdomain of $\Rn$ such that $\overline{\Rn}\setminus G$ is a nondegenerate continuum. Then for $x\,,y\in G\,,x\neq y$,
$$\mu_G(x,y)\ge \tau_n\left(\frac{\min\{d(x,\partial G),d(y,\partial G)\}}{|x-y|}\right)\ge \tau_n\left(\frac{d(x,\partial G)}{|x-y|}\right)\,.$$
\end{lem}

\begin{lem}\label{exs:8.85avv}\cite[Exercises 8.85]{avv}
Let $G$ be a proper subdomain of $\Rn$, let $x\in G$ and $B_x=B^n(x,d(x,\partial G))$. For $y\in B_x\,,x\neq y$,
$$\mu_G(x,y)\le \mu_{B_x}(x,y) = \gamma_n\left(\frac{d(x,\partial G)}{|x-y|}\right)\,.$$
\end{lem}

\begin{lem}\label{pro6.2}
For $K\geq 1$, $n\geq 2$, let $\alpha=K^{1/(1-n)}$.\\
(1) If $t_0\in (0,1)$ satisfies $\lambda_{n}^{2-\alpha}{t_0}^{\alpha}=1/2$, then for $t\in (0,t_0]$ there holds
$$\Theta_{K,n}(t)\leq 2\lambda_n^{2-\alpha}t^{\alpha}\,.$$\\
(2) If $t_1\in (0,t_0]$ satisfies $\lambda_n^{2-\alpha}{t_1}^{\alpha}=1/4$, then for $t\in (0,t_1]$ there holds
$$\Theta_{K,n}(t)\leq \frac 12\,.$$
\end{lem}

\begin{proof}
(1) By \eqref{phikn} we have
$$\varphi_{2^{n-1}K,n}^2(t)\leq \lambda_n^{2-\alpha}t^{\alpha}\,,$$
and hence by \eqref{ttg} for $t\in (0,t_0]$,
$$\Theta_{K,n}(t)\leq 2\lambda_n^{2-\alpha}t^{\alpha}\,.$$

(2) This claim follows immediately from part (1).
\end{proof}

\begin{thm}\label{qc}
Let $D, D'$ be two proper subdomains of $\Rn$  such that $\partial{D}$ is connected. Let $f:D\rightarrow D'=fD$ be a $K$-quasiconformal mapping. Then for all $x,y\in D$,
$$s_{D'}(f(x),f(y))\leq C_1 s_{D}(x,y)^{\alpha},\, \alpha=K^{1/(1-n)}\,,$$
where
$$C_1=\max \left\{2\lambda_n^{2-\alpha}(2+t_0)^{\alpha}, \left(\frac{2+t_0}{t_0}\right)^{\alpha}\right\}\,,$$
and $t_0$ is as in Lemma \ref{pro6.2} (1).
\end{thm}

\begin{proof}
The result is trivial for $x=y$. Therefore we only need to prove the theorem for $x\neq y$.

We first consider the case $|x-y|\le t_0 d(x,\partial D)$ for $x\in D$.
It is easy to see that $\partial D' = \partial (fD)$ is connected because $f$ is a homeomorphism and $\partial D$ is connected.
Therefore, there exists an unbounded domain $G'$ such that $D' \subset G'$, $d(f(x), \partial G') =d(f(x), \partial D')$ and $\partial G'$ is connected. Then for $t \ge d(f(x), \partial G')$, we have
$$
S^{n-1}(f(x),t)\cap \partial G' \neq \emptyset \,,
$$
and hence $\overline{\Rn}\setminus G'$ is a nondegenerate continuum.
By Lemma \ref{lem:8.86avv}, we have
$$\tau_n\left(\frac{d(f(x),\partial G')}{|f(x)-f(y)|}\right)\le \mu_{G'}(f(x),f(y))\le \mu_{fD}(f(x),f(y)) .$$
Let $B_x=B^n(x,d(x,\partial D))$. Then $y\in B_x$ and by Lemma \ref{exs:8.85avv},
$$\mu_D(x,y)\le\mu_{B_x}(x,y)=\gamma_n\left(\frac{d(x,\partial D)}{|x-y|}\right)\,.$$
Combining the above two inequalities, we get
$$\tau_n\left(\frac{d(f(x),\partial D')}{|f(x)-f(y)|}\right)\le \mu_{fD}(f(x),f(y))\le K\mu_D(x,y)\le K\gamma_n\left(\frac{d(x,\partial D)}{|x-y|}\right)\,,$$
and hence by Lemma \ref{pro6.2} (1),
$$s_{D'}(f(x),f(y))\le\frac{|f(x)-f(y)|}{d(f(x),\partial D')}\le \Theta_{K,n}\left(\frac{|x-y|}{d(x,\partial D)}\right) \le 2\lambda_n^{2-\alpha}\left(\frac{|x-y|}{d(x,\partial D)}\right)^{\alpha}\,.$$
On the other hand,
 $$s_D(x,y)\geq \frac{|x-y|}{2\min\{d(x,\partial D), d(y,\partial D)\} +|x-y| }\geq \frac{|x-y|}{(2+t_0) d(x,\partial D)}.$$
Therefore, for $|x-y|\le t_0 d(x,\partial D)$,
\bequu
s_{D'}(f(x),f(y)) &\leq & 2\lambda_n^{2-\alpha}\left((2+ t_0)s_D(x,y)\right)^{\alpha}\\
&=& 2\lambda_n^{2-\alpha}(2+ t_0)^{\alpha}s_D(x,y)^{\alpha}.
\eequu

Now it only remains to prove the case $|x-y|>t_0 d(x,\partial D)$ for $x\,,y\in D$. We easily see that
$$s_D(x,y)\geq \frac{|x-y|}{2|x-y|/{t_0}+|x-y|}=\frac{t_0}{2+t_0}\,,$$
and hence
$$s_{D'}(f(x),f(y))\leq 1 \leq \left(\frac{2+t_0}{t_0}\right)^{\alpha} s_D(x,y)^{\alpha}\,.$$

Thus we complete the proof by choosing the constant
$$C_1=  \max \left\{2\lambda_n^{2-\alpha}(2+t_0)^{\alpha}, \left(\frac{2+t_0}{t_0}\right)^{\alpha}\right\}.$$

\end{proof}

\begin{rem}
In Theorem \ref{qc} the hypothesis of $f$ being a quasiconformal map cannot be replaced with an analytic function. To see this, we consider the analytic function $g:\mathbb{B}^2\rightarrow \mathbb{B}^2\setminus\{0\}=g\mathbb{B}^2$ with $g(z)=\exp\left(\frac{z+1}{z-1}\right)$.
Let $r_k=\frac{k-1}{k+1}\,(k=1,2,3\cdots)$. Then as $k\rightarrow \infty$,

$$s_{\mathbb{B}^2}(r_k,r_{k+1})=\frac{1}{3+2k}\rightarrow 0\,,$$
while
$$s_{\mathbb{B}^2\setminus\{0\}}(g(r_k),g(r_{k+1}))=\frac{e-1}{e+1}\,.$$

\end{rem}

\begin{lem}\label{lem3.6}
(1) Let the three points $x\,,y\,,z\in\Hn$ be on the line orthogonal to $\partial\Hn$ and $x_n<y_n<z_n$. Then there holds
$$v_{\mathbb{H}^n}(x,y)<v_{\mathbb{H}^n}(x,z)\,.$$
(2) Let $\lambda\in (0,1)$ and $e_n=(0, 0, \cdots,1)\in\Rn$. Let $x\in\mathbb{H}^n$ and $y\in S^{n-1}(x,\lambda x_n)$. Then for $y'=x+\lambda x_ne_n$, there holds
$$v_{\mathbb{H}^n}(x,y)\geq v_{\mathbb{H}^n}(x,y')=\arctan\frac{\lambda}{2\sqrt{1+\lambda}}> \frac{\lambda}{4}\,.$$
\end{lem}

\begin{proof}
(1) It is easy to see that the function $t\mapsto\arctan\frac{t-a}{2\sqrt{ta}}=\arctan\frac{1-a/t}{2\sqrt{a/t}}$ is increasing on $(a,\infty)$ for $a>0$. By \eqref{gen2},
$$v_{\Hn}(x,y)=\arctan\frac{y_n-x_n}{2\sqrt{x_ny_n}}\,.$$
Since $y_n<z_n$, we have that
$$v_{\mathbb{H}^n}(x,y)<v_{\mathbb{H}^n}(x,z)\,.$$

(2) By elementary geometry it is clear that the radius  of the circle through $x, y$ and tangent to $\partial{\mathbb{H}^n}$ is a decreasing function of $\theta=\measuredangle(y',x,y)\in [0,\pi]$. Hence
$$v_{\mathbb{H}^n}(x,y)\geq v_{\mathbb{H}^n}(x,y')\,.$$
By \eqref{gen2} and \eqref{ineq:th},
$$v_{\mathbb{H}^n}(x,y')=\arctan\frac{\lambda}{2\sqrt{1+\lambda}}
\ge\frac{\lambda}{4}\frac{\pi}{2\sqrt{1+\lambda}}> \frac{\lambda}{4}\,.$$
\end{proof}

\begin{thm}\label{qcv}
Let $D, D'$ be two proper subdomains of $\Rn$ such that $D$ is convex.  Let $f:D\rightarrow D'=fD$ be a $K$-quasiconformal mapping. Then for all $x\,,y\in D$,
$$v_{D'}(f(x),f(y))\leq C_2 v_{D}(x,y)^{\alpha},\, \alpha=K^{1/(1-n)},$$
where
$$C_2=\max \left\{2^{3+2\alpha}\lambda_n^{2-\alpha}, \pi\left(\frac{4}{t_1}\right)^{\alpha}\right\}\,,$$
and $t_1$ is as in Lemma \ref{pro6.2} (2).
\end{thm}

\begin{proof}
The result is trivial for $x=y$. Therefore we only need to prove the theorem for $x\neq y$.

We first consider the case $|x-y|\le t_1 d(x,\partial D)$ for $x\in D$.
The boundary $\partial D$ is connected because $D$ is convex. By the proof of Theorem \ref{qc} and Lemma \ref{pro6.2}~(2), we have
$$\frac{|f(x)-f(y)|}{d(f(x),\partial D')}\leq \Theta_{K,n}\left(\frac{|x-y|}{d(x,\partial D)}\right)\leq \frac{1}{2}\,,$$
and hence $f(y)\in \mathbb{B}^n(f(x),d(f(x),\partial D')/2)$.

Without loss of generality, we may assume that $f(x)=0$ and $d(f(x),\partial D')=1$.  Then by \eqref{vrhoj}, we have
$$v_{D'}(f(x),f(y))\leq v_{\mathbb{B}^n}(f(x),f(y))\le\rho_{\mathbb{B}^n}(f(x),f(y))\le 2j_{\mathbb{B}^n}(f(x),f(y)).$$
Because $d(f(y),\partial D')\geq d(f(x),\partial D')/2$ and $\log(1+a)\leq a$ for $a\geq 0$, we have
$$j_{\mathbb{B}^n}(f(x),f(y))\leq 2\frac{|f(x)-f(y)|}{d(f(x),\partial D')}.$$
Therefore, the above three inequalities, combined with Lemma \ref{pro6.2}~(1), yield
$$v_{D'}(f(x),f(y)) \leq  4\Theta_{K,n}\left(\frac{|x-y|}{d(x,\partial D)}\right) \leq  8\lambda_n^{2-\alpha}\left(\frac{|x-y|}{d(x,\partial D)}\right)^{\alpha}\,.$$
Since $y\in S^{n-1}(x,|x-y|)$, then by Lemma \ref{lem3.6} (2), we have
\bequu
v_{D'}(f(x),f(y)) &\leq & 8\lambda_n^{2-\alpha}\left(\frac{|x-y|}{4d(x,\partial D)}\right)^{\alpha}4^{\alpha}\\
&\leq & 2^{3+2\alpha}\lambda_n^{2-\alpha}(v_H(x,y))^{\alpha}\,,
\eequu
where $H$ is the half space which contains $D$ such that $d(x,\partial H)=d(x,\partial{D})$. Therefore, for $|x-y|\le t_1 d(x,\partial D)$ we have
$$v_{D'}(f(x),f(y))\leq 2^{3+2\alpha} \lambda_n^{2-\alpha}(v_{D}(x,y))^{\alpha}.$$

It remains to prove the case $|x-y|>t_1 d(x,\partial D)$ for $x\,,y\in D$.
Again by Lemma \ref{lem3.6}, we have
$$v_{D}(x,y)\ge v_{H}(x,y)\ge \frac{t_1}{4}\,,$$
and hence
$$v_{D'}(f(x),f(y))\leq \pi \leq \pi\left(\frac{4}{t_1}\right)^{\alpha}v_{D}(x,y)^{\alpha}\,.$$

Thus we complete the proof by choosing the constant
$$C_2=\max \left\{2^{3+2\alpha}\lambda_n^{2-\alpha}, \pi\left(\frac{4}{t_1}\right)^{\alpha}\right\}\,.$$
\end{proof}

{\bf Acknowledgements.}
The research of the first author was supported by a grant from CIMO.
The research of the third author was supported by Academy of Finland project no. 268009.
The authors are indebted to Xiaohui Zhang for discussions on section 3, and the referee for valuable suggestions.


\end{document}